\documentclass[12pt]
{amsart}
\theoremstyle{plain}
\newtheorem{thm}{Theorem}[section]
\newtheorem{cor}[thm]{Corollary}
\newtheorem{lemma}[thm]{Lemma}

\theoremstyle{definition}

\newtheorem{exx}[thm]{Example}

\theoremstyle{remark}

\newtheorem*{remark}{Remark}

\numberwithin{equation}{section}

 \newcommand{\al}{\alpha}
  \newcommand{\vep}{\varepsilon}

\def\bR{\hbox{$\mathbb R$}}
\def\bvp{boundary value problem}
\def\beq{\begin{equation}}
\def\endeq{\end{equation}}
\def\bethm{\begin{thm}}
\def\enthm{\end{thm}}

\begin{document}

\title[periodic boundary value problems]{Existence of solutions to nonlinear second order periodic boundary value problems} 

\author{Yong Zhang}

\address{ Department of Mathematics\\ 
University of Manitoba\\ \mbox{Winnipeg R3T 2N2}\\ \indent Canada}
\email{zhangy@cc.umanitoba.ca}

\thanks{Supported by NSERC 238949-2011. }

\begin{abstract}
We consider existence of periodic boundary value problems of nonlinear second order ordinary differential equations.  Under certain half Lipschitzian type conditions several existence results are obtained. As applications positive periodic solutions of some $\phi$-Laplacian type equations and Duffing type equations are investigated.
\end{abstract}

\maketitle
\begin{center}
\today
\end{center}

\section{Introduction}

The lower and upper solution technique has been extensively investigated in studying boundary value problems of nonlinear differential equations. The ideal dates back to the work of  Perron \cite{Perron} on the Dirichlet problem for harmonic functions. The theory  for treating boundary value problems of second order nonlinear ordinary differential equations was established in the late 1960s by Jackson \cite{Jackson 68}. Significant contributions were also made by Schmitt in \cite{Schmitt 67, Schmitt 68, Schmitt 70}, and Jackson and Schrader in \cite{J-S 67}. Some notable applications were given in \cite{11, 12} by the author who characterized the existence of a positive solution to the Dirichlet problem of a type of sublinear differential equations.  Related results may also be seen in \cite{13, 14}). 

In this paper we consider the general periodic boundary value problem
\begin{equation}\label{1.1}
 \begin{cases}
 -x'' = f(t,x,x') \qquad (t\in [0,1]) \\
 x(0)=x(1), \, x'(0) = x'(1),
 \end{cases}
 \end{equation}
where $f \in C\left([0,1]\times \bR^2, \bR\right)$, the space of all real-valued continuous functions on $[0,1]\times \bR^2$. Classical results concerning this problem by using lower and upper solutions are due to Schmitt \cite{Schmitt 67}, who investigated the existence of a solution under a Nagumo condition. Recent relevant results may be found in \cite{K-S, Tor}. On the other hand, periodic solutions of special type of equations like Duffing equations have been attracting great attention lately (see \cite{F-M-N, L-S, R-T-V}).

In this paper we start with a sort of half Lipschitzian type conditions to establish an existence theorem. Then we will consider some one-sided $y$-growth conditions on the function $f(t,x,y)$  to obtain some more existence results. As applications we will consider the forced pendulum equations with curvature, that is the $\phi$-Laplacian type equations of the form
\[
(\frac{x'}{\sqrt{1+x'^2}})' + \mu(t) \sin x - \ell(t,x) x' = e(t),
\]
and Duffing type equations of the form
\[
x'' +p(t)g(x) -q(t)h(x') = e(t),
\]
deriving several existence results about positive periodic solutions of the equations.

\section{Upper and lower solutions}

Let $\al(t), \beta(t) \in C^2[0,1]$. We call $\al(t)$ a \emph{lower solution} and $\beta(t)$ an upper solution of problem (\ref{1.1}) if 
\[
-\al''(t) \leq f(t,\al, \al') \quad (t\in [0,1]), \quad \al(0) = \al(1);
\]
\[
-\beta''(t) \geq f(t,\beta, \beta') \quad (t\in [0,1]), \quad \beta(0) = \beta(1).
\]
Given $a>0$ and $b\in \bR$ there is a unique solution $h(t)$ to the linear problem
\begin{equation}\label{1.2}
 \begin{cases}
 -h'' = -ah+bh' \qquad (t\in [0,1]) \\
 h(1)-h(0)=0, \, h'(1) - h'(0) =1.
 \end{cases}
 \end{equation}
 Precisely
 \[
 h(t) = \frac{(1-e^{\lambda_2})e^{\lambda_1t} + (e^{\lambda_1}-1)e^{\lambda_2t}}{(\lambda_1-\lambda_2)(e^{\lambda_1}-1)(1-e^{\lambda_2})},
 \]
 where $\lambda_1, \lambda_2$ are respectively the positive and negative roots of the equation $\lambda^2 +b \lambda - a =0$, that is
 \[ \lambda_1 =\frac{-b+\sqrt{b^2+4a}}{2} \text{ and } \lambda_2= \frac{-b-\sqrt{b^2+4a}}{2}. \]
 Clearly, $h(t) > 0$ for $t\in [0,1]$ and $h'(0)<0$.
 
 \begin{lemma}\label{lemma 1}
 Let $\al(t)$ and $\beta(t)$ be, respectively, lower and upper solutions of problem (\ref{1.1}). Let $r_1 = \al'(0) -\al'(1)$ and $r_2 = \beta'(1) -\beta'(0)$. Suppose that 
 \begin{itemize}
 \item[($E_0$)] there are constants $a,b, \delta\in \bR$ with $a>0$ such that 
 \[
 f(t,\beta, \beta') - f(t,\al, \al') \geq -a(\beta -\al) + b(\beta' - \al') -\delta \quad (t\in [0,1])\\\\\\\\\\\ .
 \]
 \end{itemize}
 Then 
 \[ \left( \beta(t) - \al(t) \right) -(r_1 + r_2)h(t) +\delta/a \geq 0 \quad t\in [0,1]. \]
 \end{lemma}
 
 \begin{proof}
 Let 
\[\theta(t) = \left( \beta(t) - \al(t) \right) -(r_1 + r_2)h(t) +\delta/a.
\]
 Then $\theta(0)=\theta(1)$ and  $\theta'(0) = \theta'(1)$. Moreover
 \begin{align}\label{1.3}
 \theta'' &\leq -\left(f(t,\beta, \beta') - f(t,\al, \al') \right)- (r_1 + r_2)h''(t)\\ \notag
          & \leq a\left((\beta - \al) -(r_1 + r_2)h\right) - b\left((\beta' - \al') -(r_1 + r_2)h'\right) + \delta \\ \notag
          & = a \theta(t) -b\theta'(t)
 \end{align}
 for $t\in [0,1]$
 If the conclusion of the lemma were not true, then there would exist $t_1\in [0,1)$ such that $\theta(t_1) < 0$, $\theta'(t_1) =0$, and $\theta''(t_1) \geq 0$. But on the other hand, if this is the case, by inequality (\ref{1.3}) we would derive $\theta''(t_1)\leq a\theta(t_1) <0$. This is a contradiction. Thus $\theta(t)\geq 0$ for all $t\in [0.1]$.
 \end{proof}

 \begin{lemma}\label{lemma 2}
 Suppose that condition~($E_0$) holds for some $a,b,\delta\in \bR$ with $a>0$. Let 
 \[
 k_0 = -\frac{h'(0)}{h(0)} = -\frac{\lambda_1(1-e^{\lambda_2}) + \lambda_2(e^{\lambda_1}-1)}{(e^{\lambda_1}-e^{\lambda_2})}.
 \]
 Then $k_0 > 0$ and
 \begin{equation}\label{1.4}
 (\beta' - \al') +k_0(\beta-\al) -(r_1 + r_2)(k_0h+h') + k_0\delta/a \geq 0 \quad t\in [0,1].
 \end{equation}
 \end{lemma}
 
 \begin{proof}
 Since $h'(0) < 0$ it is evident that $k_0 >0$. Let $\theta$ be the same function as defined in the proof of Lemma~\ref{lemma 1}. Then (\ref{1.3}) is valid and $\theta(0)=\theta(1)$,  $\theta'(0) = \theta'(1)$. Let 
\[
\theta_1(t) = \theta'(t) + k_0\theta(t).\]
 Then $\theta_1(0) = \theta_1(1)$. Inequality~(\ref{1.4}) is tantamount to $\theta_1(t) \geq 0$ ($t\in [0,1]$). Since 
 \[
 \theta_1'(t) = \theta''(t) + k_0\theta'(t)\leq a\theta(t) + (k_0-b)\theta'(t) = \left(a-k_0(k_0 -b)\right)\theta(t) + (k_0-b)\theta_1(t),
 \]
 letting 
\[
u(t) = \theta_1'(t) - \left(a-k_0(k_0 -b)\right)\theta(t) - (k_0-b)\theta_1(t),
\]
 we have  $u(t) \leq 0$ ($t\in [0,1]$). The functions $\theta(t)$ and $\theta_1(t)$ satisfy
 \begin{equation}\label{1.5}
 \begin{cases}
 \theta' = -k_0\theta + \theta_1  \\
 \theta_1' = \left(a-k_0(k_0 -b)\right)\theta + (k_0-b)\theta_1 + u(t)  \\
 \theta(0)=\theta(1), \, \theta_1(0) = \theta_1(1).
 \end{cases}
 \end{equation}
 The problem (\ref{1.5}) has a unique solution that may be expressed as
 \begin{equation}\label{1.6}
 \left(
 \begin{matrix}
 \theta(t)\\
 \theta_1(t)
 \end{matrix} \right)
 = \left(I -A(1)\right)^{-1}\left(\int_0^t{A(t-s)} + \int_t^1{A(1+t-s)}\right)
 \left(
 \begin{matrix}
 0\\
 u(s)
 \end{matrix} \right)ds,
 \end{equation}
 where $I$ is the $2\times 2$ unit matrix and $A(t)$ is the fundamental matrix of the system~(\ref{1.5}), that is
 \[
 A(t) = \frac{1}{\lambda_1 - \lambda_2}\left(
 \begin{smallmatrix}
 (\lambda_1 +k_0)e^{\lambda_2t} - (\lambda_2 + k_0)e^{\lambda_1t} && e^{\lambda_1t} - e^{\lambda_2t}\\
 -(\lambda_1 +k_0)(\lambda_2 + k_0)(e^{\lambda_1t} - e^{\lambda_2t}) && (\lambda_1 +k_0)e^{\lambda_1t} - (\lambda_2 + k_0)e^{\lambda_2t} 
 \end{smallmatrix}\right).
 \]
 It is readily verified that
 \[
 \left(I -A(1)\right)^{-1} = \left(\begin{matrix}
 \frac{e^{\lambda_1} + e^{\lambda_2} -2}{(e^{\lambda_1}-1)(1-e^{\lambda_2})} & -\frac{e^{\lambda_1}-e^{\lambda_2}}{(\lambda_1-\lambda_2)(e^{\lambda_1}-1)(1-e^{\lambda_2})}\\
 -\frac{\lambda_1 -\lambda_2}{e^{\lambda_1} - e^{\lambda_2}} & 0 
 \end{matrix}\right).
 \]
 With these expressions one easily calculates to get
 \begin{align*}
 \theta_1(t) &= \frac{-1}{e^{\lambda_1} - e^{\lambda_2}}\left(\int_0^t{(e^{\lambda_1(t-s)} - e^{\lambda_2(t-s)})} + \int_t^1{(e^{\lambda_1(1+t-s)} - e^{\lambda_2(1+t-s)})}\right)u(s) ds\\
  &\geq 0.
 \end{align*} 
 \end{proof} 
 
 We now define
 \[
 \al_1(t) = \al(t) +r_1h(t) - \delta/(2a), \quad \beta_1(t) = \beta(t) -r_2h(t) + \delta/(2a),
 \]
 \[
 \psi_1(\eta,t) = \al_1'(t) - k_0(\eta-\al_1(t)) = \al'- k_0(\eta-\al) +r_1(k_0h+h') - k_0\delta/(2a),
 \]
 \[
 \psi_2(\eta,t) = \beta_1'(t) + k_0(\beta_1(t)-\eta) = \beta'+ k_0(\beta-\eta) - r_2(k_0h+h') + k_0\delta/(2a).
 \]
 
 \begin{lemma}\label{lemma 3}
 Suppose that ($E_0$) holds. Then
 \[
 A=\{\eta\in C^1[0,1]:\; \al_1\leq \eta \leq \beta_1, \psi_1(\eta(t),t)\leq \eta'(t) \leq \psi_2(\eta(t),t), t\in [0,1]\}
 \]
 is a nonempty bounded closed convex subset of the Banach space $C^1[0,1]$. 
 \end{lemma}
 \proof
 We only need to show $A\neq \emptyset$. The others are evident. 
 
 From Lemma~\ref{lemma 1} $\al_1 \leq \beta_1$. Let $\eta = \beta_1$. Then $\al_1\leq \eta \leq \beta_1$. Moreover, 
 \[ \eta'(t) + k_0 \eta(t) = \beta_1'(t) + k_0\beta_1(t).\]
 So $\eta'(t) = \psi_2(\eta,t)$. On the other hand, by Lemma~\ref{lemma 2} $\psi_1(\eta,t)\leq \psi_2(\eta,t)$. We therefore derive
 \[
 \psi_1(\eta(t),t)\leq \eta'(t) \leq \psi_2(\eta(t),t) \quad (t\in [0,1]).
 \]
 Thus $\eta\in A$ and hence $A\neq \emptyset$.
 \qed
 
 We note that, since $a>0$ and $k_0>0$, the set $A$ will increase if we replace $\delta/2$ with a larger number. As a consequence the following is true.
 
 \begin{cor}\label{Del}
 Suppose that ($E_0$) holds. Then for any $\Delta \geq \delta/2$ the set
 \[
 A_\Delta=\{\eta\in C^1[0,1]:\; \overline\al_1\leq \eta \leq \overline\beta_1, \overline\psi_1(\eta(t),t)\leq \eta'(t) \leq \overline\psi_2(\eta(t),t), t\in [0,1]\}
 \]
 is a nonempty bounded closed convex subset of $C^1[0,1]$, where
  \[
 \overline\al_1(t) = \al(t) +r_1h(t) - \Delta/a, \quad \overline\beta_1(t) = \beta(t) -r_2h(t) + \Delta/a,
 \]
 \[
 \overline\psi_1(\eta,t) = \overline\al_1'(t) - k_0(\eta-\overline\al_1(t)),
 \quad
 \overline\psi_2(\eta,t) = \overline\beta_1'(t) + k_0(\overline\beta_1(t)-\eta).
 \]
 \end{cor}
 
 \section{Existence Results}
 
 For $\eta\in C^1[0,1]$, consider the linear problem
 \begin{equation}\label{2.1}
 \begin{cases}
 -x'' = f(t,\eta,\eta') -a(x-\eta) + b(x'-\eta') \qquad (t\in [0,1]) \\
 x(0)=x(1), \, x'(0) = x'(1).
 \end{cases}
 \end{equation}
 Let $a,b\in \bR$ with $a>0$. Then problem~(\ref{2.1}) has a unique solution $x(t)$ and
 \beq\label{2.2}
 x(t) =(\int_0^th(t-s) + \int_t^1h(1+t-s))\left(f(s,\eta(s),\eta'(s)) + a\eta(s) -b\eta'(s)\right)ds,
 \endeq
 \beq\label{2.2'}
 x'(t) =(\int_0^th'(t-s) + \int_t^1h'(1+t-s))\left(f(s,\eta(s),\eta'(s)) + a\eta(s) -b\eta'(s)\right)ds.
 \endeq
 Formula~(\ref{2.2}) defines a continuous mapping $T$: $\eta(t) \mapsto x(t)$ on $C^1[0,1]$. It is clear that problem~(\ref{1.1}) has a solution if and only if the mapping $T$ has a fixed point in $C^1[0,1]$. Our investigation is under the following hypothesis
 \begin{itemize}
 \item[($E_1$)] There are a lower solution $\al(t)$ and an upper solution $\beta(t)$ of (\ref{1.1}) such that ($E_0$) holds, and there is $\Delta \geq \delta/2$ such that, for $\eta\in A_\Delta$ and $t\in [0,1]$,
 \[f(t,\eta, \eta') + \al''(t) \geq -a(\eta -\al) + b(\eta' - \al') -\Delta, \]
 \[f(t,\eta, \eta') + \beta''(t) \leq a(\beta -\eta) - b(\beta' - \eta') +\Delta. \]
 \end{itemize}
  In particular, ($E_1$) is satisfied if the following condition holds.
 \begin{itemize}
 \item[($E'_1$)] There are a lower solution $\al(t)$ and an upper solution $\beta(t)$ of (\ref{1.1}) such that ($E_0$) holds, and there is $\Delta \geq \delta/2$ such that, for $\eta\in A_\Delta$ and $t\in [0,1]$,
 \[f(t,\eta, \eta') - f(t,\al, \al') \geq -a(\eta -\al) + b(\eta' - \al') -\Delta, \]
 \[f(t,\beta, \beta') - f(t,\eta, \eta') \geq -a(\beta -\eta) + b(\beta' - \eta') -\Delta. \]
  \end{itemize}
 
 \bethm\label{thm 1}
 If ($E_1$) holds then the problem~(\ref{1.1}) has a solution $x(t)$ in $A_\Delta$.
 \enthm
  \proof
  From Corollary~\ref{Del} $A_\Delta \neq \emptyset$ and is 
  bounded, closed and convex. We show that the mapping $T$ is a completely continuous operator on $A_\Delta$. Then by the Schauder's fixed point theorem, $T$ has a fixed point in $A_\Delta$ and we will be done. Complete continuity of $T$ may be easily seen from the formulas~(\ref{2.2}) and (\ref{2.2'}). We now show $T(A_\Delta)\subset A_\Delta$.
  
  From (\ref{1.2}) $h(t) > 0$ on $[0,1]$. We claim 
  \beq\label{2.3}
   h'(t) + k_0 h(t) \geq 0 \quad (t\in [0,1]). \endeq
Let 
\[h_1(t) = \frac{h'(t)}{h(t)} = \frac{\lambda_1(1-e^{\lambda_2})e^{\lambda_1t} + \lambda_2(e^{\lambda_1}-1)e^{\lambda_2t}}{(1-e^{\lambda_2})e^{\lambda_1t} + (e^{\lambda_1}-1)e^{\lambda_2t}}.
\]
 Then 
\[ 
h_1'(t) = \frac{(\lambda_1-\lambda_2)^2(1-e^{\lambda_2})(e^{\lambda_1}-1)e^{(\lambda_1+\lambda_2)t}}{[(1-e^{\lambda_2})e^{\lambda_1t} + (e^{\lambda_1}-1)e^{\lambda_2t}]^2} >0 \quad
 (t\in [0,1]). 
 \]
 So $h_1(t)\geq h_1(0) = -k_0$  ($t\in [0,1]$). This shows that our claim (\ref{2.3}) holds. Define
 \[
 m(t) = (\int_0^th(t-s) + \int_t^1h(1+t-s))\left(-\al''(s)) + a\al(s) -b\al'(s) - \Delta\right)ds,
 \]
 \[
 M(t) = (\int_0^th(t-s) + \int_t^1h(1+t-s))\left(-\beta''(s)) + a\beta(s) -b\beta'(s) + \Delta\right)ds,
 \]
\[
 n(t) = [\int_0^t(h'(t-s) + k_0h(t-s)) + \int_t^1(h'(1+t-s) + k_0h(1 + t-s))]\left(-\al''(s)) + a\al(s) -b\al'(s) - \Delta\right)ds,
 \]
\[
 N(t) = [\int_0^t(h'(t-s) + k_0h(t-s)) + \int_t^1(h'(1+t-s) + k_0h(1 + t-s))]\left(-\beta''(s)) + a\beta(s) -b\beta'(s) + \Delta\right)ds.
 \]
 Then, from ($E_1$), any $\eta\in A_\Delta$ satisfies
\[
m(t) \leq T\eta (t) \leq M(t), \quad n(t) \leq (T\eta)'(t) + k_0 T\eta (t) \leq N(t) \quad (t\in [0,1]).
\]
On the other hand, one can also easily check that the following identities hold.
\[
m(t) = \al(t) +r_1h(t) - \Delta/a =  \overline\al_1(t), \quad M(t) = \beta(t) -r_2h(t) + \Delta/a = \overline\beta_1(t),
 \]
 \[
 n(t) = \al'(t)+ k_0\al(t) +r_1(k_0h(t)+h'(t)) - k_0\Delta/a = \overline\psi_1(T\eta,t) + k_0T\eta(t),
 \]
\[
 N(t) = \beta'(t)+ k_0\beta(t) - r_2(k_0h(t)+h'(t)) + k_0\Delta/a = \overline\psi_2(T\eta,t) + k_0T\eta(t).
 \]
 These finally lead to $T\eta \in A_\Delta$. Therefore we have shown $T(A_\Delta)\subset A_\Delta$. The proof is complete.
\qed

We note that in Theorem~\ref{thm 1} we do not require $\al(t) \leq \beta(t)$ and we do not require $r_1\geq 0$ or $r_2 \geq 0$.  We also note that, if there are lower and upper solutions $\al(t)$ and $\beta(t)$ for problem~(\ref{1.1}), one can always choose the numbers $a,b$ and $\delta$ so that condition~($E_0$) is valid. The problem is whether we can choose them so that ($E_1$) is satisfied. In the sequel we will focus on the case that the lower and upper solutions $\al(t)$ and $\beta(t)$ satisfy $\al(t) \leq \beta(t)$, $r_1\geq 0$ and $r_2 \geq 0$.

\begin{exx}\label{example sin} (Forced pendulum with curvature operator)
Consider the $\phi$-Laplacian periodic boundary value problem
\begin{equation}\label{sin -}
\begin{cases}
(\frac{x'}{\sqrt{1+x'^2}})' + \mu(t) \sin x - \ell(t,x) x' = e(t) \\
x(0) = x(1), x'(0) = x'(1),
\end{cases}
\end{equation}
where $\mu(t),e(t)\in C[0,1]$, $\ell(t,x)\in C([0,1]\times \bR)$. Suppose that $\mu(t) \geq |e(t)|$ for $t\in[0,1]$. This problem has constant lower and upper solutions $\al =\pi/2$ and $\beta= 3\pi/2$. Let 
\[
\ell_0(t) = \min\{\ell(t,x):\; \pi/2\leq x\leq 3\pi/2 \}.
\]
 We show that if there is $r>0$ such that $r\ell_0(t) \geq \mu(t)$ ($t\in [0,1]$), then there is a solution of (\ref{sin -}) between $\al$ and $\beta$.
\end{exx}

\proof
We rewrite the problem in the standard form (\ref{1.1}). The right side function is indeed
\[
f(t,x,y) = (1+y^2)^{3/2}\left[\mu(t) \sin x - \ell(t,x) y - e(t) \right].
\]
Let $d$ be a constant such that $d\geq (1+\pi^2r^2)^{3/2}\ell_{max}$, where \[
\ell_{max} = \max\{\ell(t,x), \; t\in [0,1], \pi/2 \leq x\leq 3\pi/2\}.
\]
 We choose $a=rd$ and $b=-d$. Clearly, $(E_0)$ is satisfied with $\delta =0$. We have
\[
k_0 = -\lambda_2 - \frac{(\lambda_1-\lambda_2)(1-e^{\lambda_2})}{(e^{\lambda_1}-e^{\lambda_2})}\leq -\lambda_2
 = \frac{2a}{-b + \sqrt{b^2 + 4a}} =  \frac{2rd}{d + \sqrt{d^2 + 4rd}} \leq r.
\]
 For these choices of $a$ and $b$  we take $\Delta =0$. Then
\begin{align*}
A_\Delta = A_0 &\subset \{\eta\in C^1[0,1]: \; \pi/2 \leq \eta \leq 3\pi/2, -r(\eta - \pi/2)\leq \eta' \leq r(3\pi/2-\eta)\} \\ 
                          &\subset \{\eta\in C^1[0,1]: \; \pi/2 \leq \eta \leq 3\pi/2, -r \pi \leq \eta' \leq r\pi\}.
\end{align*}
We now show that $(E_1)$ holds. Precisely, for $\eta \in A_0$ we show
\begin{equation}\label{for alpha}
f(t,\eta, \eta') + a (\eta - \pi/2) - b \eta' \geq 0
\end{equation}
\begin{equation}\label{for beta}
 a (3\pi/2 - \eta) + b \eta' - f(t,\eta, \eta') \geq 0.
\end{equation}
Now let $\eta \in A_0$. We first note that 
\[
\mu(t) \sin \eta \geq -\mu(t) (\eta -\pi/2) + \mu(t)
\]
 simply by the mean value theorem. Using the condition $\mu(t) - e(t) \geq 0$ we have
\[
f(t,\eta, \eta') \geq (1+ \eta'^2)^{3/2} [-\mu(t) (\eta - \pi/2) - \ell(t,\eta) \eta' ].
\]
Since $b=-d$ and $d - (1+ \eta'^2)^{3/2}\ell(t,\eta)\geq 0$ from the assumption,
\begin{align*}
&f(t,\eta, \eta') + a (\eta - \pi/2) - b \eta' \\
&\geq -(1+ \eta'^2)^{3/2} \mu(t) (\eta - \pi/2) + a (\eta - \pi/2) +[d - (1+ \eta'^2)^{3/2}\ell(t,\eta)] \eta'\\
    &\geq  -(1+ \eta'^2)^{3/2} \mu(t) (\eta - \pi/2) + a (\eta - \pi/2) -[d - (1+ \eta'^2)^{3/2}\ell(t,\eta)] r(\eta - \pi/2)\\
 &= (1+ \eta'^2)^{3/2} \left(r\ell(t,\eta) -\mu(t)\right) (\eta - \pi/2) + (a- rd) (\eta - \pi/2)  \geq 0
\end{align*}
from the choices of $r$ and $d$. Thus we have shown that (\ref{for alpha}) is valid for $\eta \in A_0$.

Similarly, one may verify that (\ref{for beta}) is also valid for $\eta \in A_0$. Therefore $(E_1)$, with $\Delta = 0$, holds for the problem. The conclusion follows from Theorem~\ref{thm 1}.
\qed 

About problem~(\ref{sin -}) we refer to \cite{M-W} and \cite{B-J-M} for relevant results.

We now consider an example for which an upper solution that one can find is not constant.

\begin{exx}\label{example1} (Singular nonlinearities of attractive type)
Let $p(t), e(t)\in C[0,1]$ 
and $\lambda>0$. Consider
\begin{equation}\label{-lambda}
\begin{cases}
x'' + p(t) x^{-\lambda} = e(t) \\
x(0) = x(1), x'(0) = x'(1).
\end{cases}
\end{equation}
If $\int_0^1{e(t)}dt>0$ and there is $C>0$ such that $Cp(t) \geq e(t)$ on $[0,1]$, then problem~(\ref{-lambda}) has a solution $x(t) >0$ ($t\in [0,1]$).
\end{exx}
\proof
Let $c>0$ be the number satisfying $c^\lambda C =1$. Then $\al(t) = c$ is a lower solution of (\ref{-lambda}). Let 
\[
n = \frac{1}{2}\int_0^1{e(t)}dt
\]
 and let $d\geq c$ be a number such that $p(t) x^{-\lambda} \leq 2n$ ($t\in [0,1]$) whenever $x\geq d$. Define
\[
\beta(t) = m + nt(1-t) -[(1-t)\int_0^t{se(s)}ds + t\int_t^1{(1-s)e(s)}ds],
\]
where $m>0$ is a constant sufficiently large so that $\min_{t\in [0,1]}\beta(t) \geq d$. Then 
\[
\beta''(t) = -2n + e(t).
\]
One may easily check that $\beta(t)$ is an upper solution of (\ref{-lambda}). Moreover $\beta\geq \al$ and 
\[
\beta('1) -\beta'(0) = -2n + \int_0^1{e(t)}dt = 0.
\]
 On the other hand
\begin{equation}\label{inequality}
p(t)x_1^{-\lambda} - p(t)x_2^{-\lambda} \geq -\frac{\lambda p_{\max}}{c^{\lambda +1}}(x_1 -x_2)
\end{equation}
for all $x_1 \geq x_2\geq c$, where $p_{\max} = \max_{t\in [0,1]}|p(t)|$.

Now set 
\[
f(t,x,y) = 
\begin{cases}
p(t) x^{-\lambda} - e(t), & \text{ if } x\geq c\\
p(t) c^{-\lambda} - e(t), & \text{ if } x< c
\end{cases} 
\]
Then $f\in C([0,1]\times \bR^2)$. With this $f$ it is evident that any solution $x(t)$ of problem~(\ref{1.1}) is also a solution of (\ref{-lambda}) if $x(t) \geq c$ ($t\in [0,1]$). Clearly, $\al$ and $\beta$ are lower and upper solutions of (\ref{1.1}). Take 
\[
a = \frac{\lambda p_{\max}}{c^{\lambda +1}}.
\]
 Inequality~(\ref{inequality}) implies that condition ($E_0$) holds with this $a$ and with $b = \delta =0$. Since 
\[
r_1 = \al'(0)-\al'(1) =0 =r_2 = \beta'(1) -\beta'(0),
\]
 taking $\Delta = 0$ we have 
\[
A_\Delta \subset \{\eta \in C^1[0,1]: \; \al \leq \eta \leq \beta\}.
\]
 Again (\ref{inequality}) shows that condition ($E'_1$) is satisfied. Thus, from Theorem~\ref{thm 1}, (\ref{1.1}) has a solution $x(t)$ between $\al$ and $\beta$. Since $x(t) \geq \al =c$, it is a solution of (\ref{-lambda}). 
\qed

\begin{remark}
Suppose that $p(t)\geq 0$ on $[0,1]$. Then the condition $\int_0^1{e(t)}dt>0$ is in fact necessary for (\ref{-lambda}) to have a positive solution. To see this one needs only to integrate the left hand side of the equation in (\ref{-lambda}), assuming that $x = x(t) \geq 0$ is a solution. In particular, when $p(t)$ is a positive constant function, then Problem~\ref{-lambda} admits a positive solution if and only if $\int_0^1{e(t)}dt>0$. This indeed  was the essential case studied by Lazer and Solimini in \cite{L-S}.
\end{remark}

We now consider special cases of Theorem~\ref{thm 1}. We consider the following conditions.

\begin{itemize}
\item[($E_2$)] Problem~\ref{1.1} has lower and upper solutions $\al(t)$ and $\beta(t)$ that satisfy $\al(t) \leq \beta(t)$, $r_1 = \al'(0)-\al'(1) \geq 0$ and $r_2 =\beta'(1) - \beta'(0) \geq 0$. The function $f(t, x, y)$ satisfies a local Lipschitz condition near the curves $(t, \al, \al')$ and $(t, \beta, \beta')$ ($t\in [0,1]$), that is that there are constants $\mu>0$ and $\ell>0$ such that 
\[
|f(t,x_1,y_1) - f(t,x_2,y_2)| \leq \ell |x_1-x_2| + \ell |y_1 - y_2|
\]
whenever $(t,x_i,y_i)\in G_{\al,\mu}$ or $(t,x_i,y_i)\in G_{\beta,\mu}$, $i=1,2$. Here, for $\gamma\in C^1[0,1]$, $G_{\gamma,\mu}$ is the subset of $\bR^3$ defined by
\[
G_{\gamma,\mu} = \{(t,x,y):\; t\in [0,1], \gamma(t) - \mu \leq x \leq \gamma(t) + \mu, \gamma'(t) - \mu \leq y \leq \gamma'(t) + \mu\}.
\] 
\end{itemize} 

\begin{itemize}
\item[($E_3$)] For the lower and upper solutions $\al(t)$ and $\beta(t)$ there are constants $L>0$ and $K>0$ and a function $c(t)\geq 0$ on $[0,1]$ with $c(t)(\beta(t) - \al(t)) \leq c < 1$ ($t\in [0,1]$) such that
\[
f(t,x,\al'(t)-z) - f(t,x,\al'(t)) \geq -c(t)z^2 - Lz -K
\]
\[
f(t,x,\beta'(t)) - f(t,x,\beta'(t)+z) \geq -c(t)z^2 - Lz -K
\]
whenever $t\in [0,1]$, $\al(t)\leq x\leq \beta(t)$ and $z\geq 0$.
\end{itemize} 

\begin{thm}\label{thm 2}
Suppose that ($E_2$) and ($E_3$) hold. Then problem (\ref{1.1}) has a solution $x(t)$ satisfying $\al(t)\leq x(t) \leq \beta(t)$ ($t\in [0,1]$).
\end{thm}
 We postpone the proof of Theorem~\ref{thm 2} to the next section. Here we look at some special cases. First, we have the following variant of Theorem~\ref{thm 2}.
 
\begin{thm}\label{thm 3}
Suppose that ($E_2$) and the following ($E'_3$) hold.
\begin{itemize}
\item[($E'_3$)] For the lower and upper solutions $\al(t)$ and $\beta(t)$ there are constants $L>0$ and $K>0$ and a function $c(t)\in C[0,1]$ with $c(t)(\beta(t) - \al(t)) \leq c < 1$ ($t\in [0,1]$) such that
\[
f(t,x,\al'(t)) - f(t,x,\al'(t)+z) \leq c(t)z^2 + Lz +K
\]
\[
f(t,x,\beta'(t)-z) - f(t,x,\beta'(t)) \leq c(t)z^2 + Lz +K
\]
whenever $t\in [0,1]$, $\al(t)\leq x\leq \beta(t)$ and $z\geq 0$.
\end{itemize} 
 Then problem (\ref{1.1}) has a solution $x(t)$ satisfying $\al(t)\leq x(t) \leq \beta(t)$ ($t\in [0,1]$).
\end{thm}
\proof
In (\ref{1.1}) we let $y(t) = x(1-t)$ and let $F(t,y,y') = f(1-t, y, -y')$. Then (\ref{1.1}) is equivalent to the \bvp 
\begin{equation}\label{2.4}
 \begin{cases}
 -y'' = F(t,y,y') \qquad (t\in [0,1]) \\
 y(0)=y(1), \, y'(0) = y'(1).
 \end{cases}
 \end{equation}
 (\ref{1.1}) has a solution $x(t)$ satisfying $\al(t)\leq x(t) \leq \beta(t)$ if and only if (\ref{2.4}) has a solution $y(t)$ satisfying $\al(1-t)\leq y(t) \leq \beta(1-t)$. If ($E_2$) holds then $\al_0(t) = \al(1-t)$ and $\beta_0(t) = \beta(1-t)$ are respectively lower and upper solutions of (\ref{2.4}) satisfying $\al_0 \leq \beta_0$, $\al'_0(0)-\al'_0(1) = r_1 \geq 0$ and $\beta'_0(1)-\beta'_0(0) = r_2 \geq 0$. So ($E_2$), with $\al, \beta$ replaced by $\al_0, \beta_0$ and $f$ replaced by $F$, is valid for problem (\ref{2.4}). We now show that ($E_3$) also holds with respect to $F$, $\al_0$ and $\beta_0$. In fact, from ($E'_3$) we derive 
 \[
 F(t,y,\al'_0(t) -z) -F(t,y,\al'_0(t))
  \geq f(1-t,y,\al'(1-t)+z) - f(1-t,y,\al'(1-t)) \geq -c(1-t)z^2 - Lz -K,
\]
\[
 F(t,y,\beta'_0(t)) -F(t,y,\beta'_0(t)+z) 
 \geq f(1-t,y,\beta'(1-t)) - f(1-t,y,\beta'(1-t)-z) \geq -c(1-t)z^2 -Lz -K
 \]
and 
\[
c(1-t)(\beta_0(t) -\al_0(t)) = c(1-t)(\beta(1-t) - \al(1-t)) \leq c <1\quad (t\in [0,1]). 
\]
This shows our conclusion. Then, using Theorem~\ref{thm 2}, we are ensured that problem~(\ref{2.4}) has a solution between $\al_0$ and $\beta_0$. The proof is complete.
\qed

\bethm\label{cor 1}
Suppose that $f = f_1 + f_2$ and ($E_2$) holds. Denote 
\[
\Gamma =\{(t,x): \; t\in [0,1], \al(t)\leq x \leq \beta(t)\}.
\]
 If there is a function $c(t)\in C[0,1]$ with $c(t)(\beta(t) - \al(t)) \leq c < 1$ ($t\in [0,1]$) such that one of the following conditions is satisfied, then the problem (\ref{1.1}) has a solution $x(t)$ satisfying $\al(t)\leq x(t) \leq \beta(t)$ ($t\in [0,1]$).
\begin{enumerate}
\item\label{1}  $\limsup_{|y|\to \infty}\frac{yf_1(t,x,y)}{|y|^3}\leq c(t)$  uniformly on $\Gamma$ and $f_2(t,x,y)$ is nonincreasing in $y$ for each $(t,x)\in \Gamma$;
\item\label{2}  $\liminf_{|y|\to \infty}\frac{yf_1(t,x,y)}{|y|^3}\geq -c(t)$  uniformly on $\Gamma$ and $f_2(t,x,y)$ is nondecreasing in $y$ for each $(t,x)\in \Gamma$.
\end{enumerate}
\enthm
\proof
From Theorems~\ref{thm 2} and \ref{thm 3} it suffices to show either ($E_3$) or ($E'_3$) holds. We show that condition~(\ref{1}) implies ($E_3$). The implication (\ref{2}) to ($E'_3$) can be shown similarly.

Suppose that condition~(\ref{1}) holds. Without loose of generality we may assume 
\[
\limsup_{y\to\pm \infty}\frac{yf_1(t,x,y)}{|y|^3} < c(t) 
\]
 uniformly on $\Gamma$ (otherwise we can increase $c(t)$ by a small number $\vep>0$ without affecting other assumptions). Then there exists $y_0 > |\al'(t)| + |\beta'(t)|$ ($t\in [0,1]$) such that 
 \[
 sign(y)f_1(t,x,y) \leq c(t) y^2 \quad ((t,x)\in \Gamma, |y|> y_0).
 \]
  Let
\begin{align*}
K &= 2 \max\{|f_1(t,x,y)|+ c(t)\left((\al'(t))^2 +(\beta'(t))^2 \right): \; (t,x)\in \Gamma, |y|\leq y_0\},\\
L &= 2 \max_{t\in [0,1]}(|\al'(t)| +|\beta'(t)|).
\end{align*}
Then for $(t,x)\in \Gamma$ and $z\geq 0$ we have
\begin{align*}
f(t,x,\al'(t)-z) - f(t,x,\al'(t)) &\geq f_1(t,x,\al'(t)-z) - f_1(t,x,\al'(t)) \\
&\geq -c(t)z^2 - Lz -K
\end{align*}
\begin{align*}
f(t,x,\beta'(t)) - f(t,x,\beta'(t)+z) &\geq f_1(t,x,\beta'(t)) - f_1(t,x,\beta'(t)+z) \\
&\geq -c(t)z^2 - Lz -K.
\end{align*}
So condition ($E_3$) is valid. The proof is complete.
\qed

As an application, we consider the following problem.
\begin{equation}\label{h(x')}
\begin{cases}
x'' + p(t) g(x) -q(t) h(x') = e(t) \\
x(0) = x(1), x'(0) = x'(1),
\end{cases}
\end{equation}
where $p(t), q(t), e(t)\in C[0,1]$ with $q(t)\geq 0$ on $[0,1]$, $g(x)\in C(0,\infty)$  
 and $h(y) \in C(\bR)$.

For $p\equiv 1$ and $q\equiv 0$ Problem~\ref{h(x')} have been studied by many authors (see for example \cite{L-S, R-T-V}).

Recall that a function $g(x)$ satisfies local Lipschitz condition on a set $I\subset \bR$ if for every compact set $K\subset I$ there is $\ell_K >0$ such that
\[
|g(x_1) - g(x_2)| \leq \ell_K |x_1 -x_2| \quad (x_1, x_2 \in K).
\]

\begin{exx}\label{example2}
  Suppose that $g$ and $h$ satisfy local Lipschitz condition on $(0,\infty)$ and $\bR$ respectively. Suppose further that
\begin{equation}\label{for g}
\limsup_{x\to 0+} g(x) =\infty, \quad \lim_{x\to \infty}g(x) = 0,
\end{equation}
and
\[
 h\geq 0, \; h(0)= 0, \;
 \lim_{y\to -\infty}h(y)/y^2 = 0. 
\]
If 
\[
\bar e = \int_0^1{e(t)}dt > 0
\]
 and there is $C>0$ such that $Cp(t) \geq e(t)$ on $[0,1]$, then problem~(\ref{h(x')}) has a solution $x(t) >0$ ($t\in [0,1]$).
\end{exx}

\proof
(\ref{for g}) implies that there is $c>0$ such that $g(c) = C$. So $\al \equiv c$ is a lower solution of (\ref{h(x')}). On the other hand 
\[
\beta(t) = m + nt(1-t) -[(1-t)\int_0^t{s\, e(s)}ds + t\int_t^1{(1-s)e(s)}ds]
\]
is an upper solution of (\ref{h(x')}) when $m>0$ is sufficiently large, where 
\[
n= \frac{1}{2}\int_0^1{e(t)}dt
\]
 as in the proof of example~\ref{example1}. Since $g$ and $h$ satisfy local Lipschitz condition, it is readily seen that 
\[
f(t,x,y) = p(t) g(x) -q(t)h(y) -e(t)
\]
 satisfies condition~($E_2$). Let 
\[
f_1(t,x,y) = -q(t)h(y)\quad  \text{and}\quad f_2(t,x,y) = p(t)g(x)- e(t).
\]
 (Here we extend the domain of $f$ to $[0,1]\times \bR^2$ as in the proof of example~\ref{example1}). Then condition~(1) of Theorem~\ref{cor 1} holds with $c(t) \equiv 0$. The result then follows from Theorem~\ref{cor 1}.
\qed

\begin{remark}
Typical functions $g$ and $h$ that satisfy the requirements in Example~\ref{example2} are 
\[
g(x) = x^{-\lambda}
\]
 and
\[
h(y) = 
\begin{cases}
|y|^{\lambda_1}, & \text{ if $y < 0$}\\
y^{\lambda_2}, & \text{ if $y \geq 0$}
\end{cases}
\]
where $\lambda >0$, $1\leq \lambda_1 < 2$ and $\lambda_2 \geq 1$. The condition $\lim_{x\to \infty}g(x) = 0$ is only used in the proof to show that $\beta$ defined there is an upper solution when $m$ is large. If $p \geq 0$ this condition may certainly be weakened to $\limsup_{x\to \infty}g(x) \leq 0$.
For example, The conclusion is still true if $p \geq 0$ and $g(x) = ax^{-\lambda} - bx^\omega$ for any positive numbers $a,b,\lambda$ and $\omega$ (in contrast with with \cite[Example~3.4]{R-T-V}). Moreover the condition $\limsup_{x\to 0+} g(x) =\infty$ is only used to guarantee the existence of $c>0$ such that $g(c) = C$. It may be removed if we simply assume that there is $c>0$ such that $g(c)p(t) \geq e(t)$ ($t\in [0,1]$).
\end{remark}

Modify the construction of $\beta$ in the proof of Example~\ref{example2} we may allow $h(y)$ to change sign. Precisely we have the following.

\begin{exx}\label{example3}
 Let $p(t)\geq 0$ $q(t)\geq 0$ and $e(t) \geq 0$ ($t\in [0,1]$), and let $g$ and $h$ satisfy local Lipschitz condition on $(0,\infty)$ and $\bR$ respectively with $\limsup_{x\to \infty}g(x) \leq 0$. Suppose that 
\[
\bar q = \int_0^1{q(t)}dt > 0, \quad \bar e = \int_0^1{e(t)}dt > 0,
\]
 and suppose that
\[
h(0)= 0, \quad
 h(y) >-\bar e/\bar q \; \text{ for }|y|\leq 2\bar e,
\]
and either 
\[
\liminf_{|y|\to \infty}yh(y)/|y|^3 \geq  0, \text{ or } \limsup_{|y|\to \infty}yh(y)/|y|^3 \leq  0
\]
holds.
 If there is $c>0$ such that $g(c)p(t) \geq e(t)$ on $[0,1]$, then problem~(\ref{h(x')}) has a solution $x(t) >0$ ($t\in [0,1]$).
\end{exx}
\proof
Still $\al \equiv c$ is a lower solution of (\ref{h(x')}).  Let $n_1>0$ be a constant such that 
\[
-h(y) \leq n_1 < \bar e/\bar q
\]
 for $|y|\leq 2 \bar e$ and let $n_2 = (\bar e - \bar q n_1)/\bar p$ $(>0)$, where $\bar p = \int_0^1{p(t)}dt$. Let
\[
\beta(t) = m+ (1-t)\int_0^t{s[n_1q(s) + n_2p(s) -  e(s)]}ds + t\int_t^1{(1-s)[n_1q(s) + n_2p(s) -  e(s)]}ds
\]
We have $\beta(0) = \beta(1) = m$,
\[
\beta'(t) = \int_t^1{[n_1q(s) + n_2p(s) -  e(s)]}ds - \int_0^1{s[n_1q(s) + n_2p(s) -  e(s)]}ds,
\]
 \[
\beta'(1) - \beta'(0) =  -(n_1\bar q + n_2 \bar p -\bar e) = 0,
\]
 and 
\[
\beta''(t) = - n_1q(t) - n_2p(t) +  e(t).
\]
Since $\limsup_{x\to \infty} g(x) \leq 0$, for sufficiently large $m$ we have $g(\beta) \leq n_2$. On the other hand
\[
 - \int_t^1{ e(s)}ds - \int_0^1{s[n_1q(s) + n_2p(s) ]}ds \leq \beta'(t) \leq  \int_t^1{(n_1q(s) + n_2p(s) )}ds + \int_0^1{s\,  e(s)}ds.
\]
This implies $|\beta'(t)| \leq 2\bar e$ for $t\in [0,1]$.
Therefore
\begin{align*}
\beta''(t) + p(t)g(\beta) - q(t)h(\beta') &= -p(t)(n_2 - g(\beta)) -q(t)(n_2 + h(\beta')) +e(t) \\
&\leq e(t) \quad (t\in [0,1]).
\end{align*}
So $\beta(t)$ is an upper solution of (\ref{h(x')}) when $m>0$ is sufficiently large. Take $f_1$, $f_2$ as in the proof of example~\ref{example2}.  Then condition of Theorem~\ref{cor 1} holds with $c(t) \equiv 0$. Thus problem~(\ref{h(x')}) has a solution $x(t) \geq c$.
\qed

\begin{remark}
an example of the function $h$ that satisfies the assumption in Example~\ref{example3} is 
\[h(y) = \mu y^{2k+1} - \nu y,
\]
 where $\mu$ and $\nu$ are constants satisfying 
\[
0<\mu \leq \frac{1}{(2\bar e)^{2k}(2\bar q)} \quad \text{and} \quad 0<\nu \leq \frac{1}{(2\bar q)}.
\]
\end{remark}

\section{Proof of theorem~\ref{thm 2}}
We first observe that from ($E_3$), when 
\[
(t,x)\in \Gamma = \{(t,x): \; t\in [0,1], \al(t) \leq x\leq \beta(t)\},
\]
 the following inequalities hold.
\begin{equation}\label{3.1}
f(t,x,\al'(t)-z) - f(t,\al(t),\al'(t)) \geq -c(t)z^2 - Lz -\hat K,
\end{equation}
\begin{equation}\label{3.2}
f(t,\beta(t),\beta'(t)) - f(t,x,\beta'(t)+z) \geq -c(t)z^2 - Lz -\hat K,
\end{equation}
where 
\[
\hat K = K + \max_{(t,x)\in \Gamma}\left(|f(t,x,\al'(t) - f(t,\al(t),\al'(t))| +|f(t,x,\beta'(t)) - f(t,\beta(t),\beta'(t))|\right).
\]

The proof may be divided into 3 steps.

Step 1. Assume 
$\al(t) < \beta(t)$, $-\al''(t) < f(t,\al, \al')$,
and $-\beta''(t) > f(t,\beta, \beta')$
 for all $t\in [0,1]$. Then there are $\vep_1 >0$, $\vep_2 >0$ and $N_0 = N_0(\vep_1)>0$ such that for $t\in [0,1]$ we have $\beta(t)-\al(t) \geq \vep_1$,
\begin{equation}\label{3.3}
 -\al''(t) \leq f(t,\al, \al') - \vep_2, \quad -\beta''(t) \geq f(t,\beta, \beta') +\vep_2,
\end{equation}
and for $N\geq N_0$
\[
\beta(t)-\al(t) +\frac{1}{N}(\beta'(t)-\al'(t)) > 0 \quad (t\in [0,1]).
\]
Given such an $N$, there is $a_0 = a_0(N)$ such that when $a\geq a_0$ and $b=-a/N +N$ we have
\begin{equation}\label{3.4}
a(\beta(t)-\al(t)) - b(\beta'(t)-\al'(t)) + f(t,\beta, \beta') - f(t,\al, \al') \geq  0 \quad (t\in [0,1]).
\end{equation}
Thus, condition~($E_0$) is valid for $\delta =0$ and for this choice of $a$ and $b$ provided $N\geq N_0$. With these $a$ and $b$ we further have
\begin{align}\label{3.5}
0< k_0 &= -\frac{\lambda_1(1-e^{\lambda_2}) + \lambda_2(e^{\lambda_1}-1)}{(e^{\lambda_1}-e^{\lambda_2})} \\
&= -\lambda_2 - \frac{(\lambda_1-\lambda_2)(1-e^{\lambda_2})}{(e^{\lambda_1}-e^{\lambda_2})}\leq -\lambda_2 = N.\notag
\end{align}
Take $\Delta =0$. Relation~(\ref{3.5}) implies
\[
A_\Delta \subset \{\eta\in C^1[0,1]:\; \al\leq \eta \leq \beta, \al'-N(\eta-\al) \leq \eta'\leq \beta'+ N(\beta-\eta)\}.
\]
Denote by $E_N$ the set of all points $(t,x,y)$ in $\bR^3$ such that $t\in [0,1]$, $\al(t) \leq x\leq \beta(t)$ and 
\[
\al'(t)-N(x-\al(t)) \leq y\leq \beta'(t) +N(\beta-x).
\]
 Then $E_N$ is a compact subset of $\bR^3$. Since $f$ is continuous on $E_N$, there is a Lipschitz function $F$ such that
\[
|F(t,x_1,y_1) - F(t,x_2,y_2)|\leq C_N(|x_1-x_2| + |y_1-y_2|) \quad ((t,x_i,y_i)\in E_N, i=1,2)
\]
and
\[
|f(t,x,y)- F(t,x,y)| \leq \vep_2/2 \quad ((t,x,y)\in E_N),
\]
where $C_N>0$ is some constant depending on $N$ and $\vep_2$ (see for instance \cite[Theorem~IV.6.16]{D-S}). This leads to
\begin{equation}\label{3.6}
f(t,x_1,y_1) - f(t,x_2,y_2)\geq -C_N|x_1-x_2| - C_N|y_1-y_2| -\vep_2 \quad ((t,x_i,y_i)\in E_N, i=1,2)
\end{equation}
Now let 
\[
N = \max\{N_0, \ell + 1, \frac{\hat K +L\mu}{(1-c)\mu}\}
\]
and let
\[
a \geq max\{a_0, (N+ C_N +\ell)N\} ,
\]
where $\ell$, $\mu$ , $c$ and $L$ are as in ($E_2$) and ($E_3$). We show
 \begin{equation}\label{3.7}
 f(t,\eta, \eta') + \al''(t) \geq -a(\eta -\al) + b(\eta' - \al') \quad (\eta\in A_\Delta, t\in [0,1]). 
 \end{equation}
 Due to (\ref{3.3}) it suffices to show 
 \begin{equation}\label{3.8}
  I(t,\eta) \geq 0 \quad (\eta\in A_\Delta, t\in [0,1]),
 \end{equation}
where
\[
 I(t,\eta) = a(\eta -\al) - b(\eta' - \al') + f(t,\eta, \eta') - f(t,\al, \al') + \vep_2.
\] 
We prove (\ref{3.8}) in 3 cases, assuming $\eta\in A_\Delta$ and $t\in [0,1]$.

Case 1. If $\eta'(t) \geq \al'(t)$, then 
\[
I(t,\eta) \geq (a-C_N)(\eta -\al) + (a/N -N -C_N)(\eta' - \al') \geq 0.
\]

Case 2. If $\eta'(t) < \al'(t)$ and $\al(t)\leq \eta(t) \leq \al(t) + \mu/N$, then $(t,\eta(t), \eta'(t))\in G_\al$. From ($E_2$) we get
\[
I(t,\eta) \geq (a-\ell)(\eta -\al) + (a/N -N +\ell)(\eta' - \al') \geq (N^2 -\ell N -\ell)(\eta -\al)\geq 0.
\]

Case 3. If $\eta'(t) < \al'(t)$ and $\eta(t) > \al(t) + \mu/N$, then from (\ref{3.1}) and ($E_3$) we get
\begin{align*}
I(t,\eta) &\geq a(\eta -\al) - b(\eta' - \al') +(c(t)|\eta' - \al'| +L)(\eta' - \al') - \hat K \\
&\geq a(\eta -\al) + (-b + c(t)N(\eta -\al) +L)(\eta' -\al') -\hat K \\
&\geq a(\eta -\al) + (-b + cN +L)(\eta' -\al') -\hat K \\
&\geq [(1-c)N^2 -LN](\eta -\al)-\hat K\geq [(1-c)N -L]\mu-\hat K \geq 0.
\end{align*}
Therefore, we have shown that (\ref{3.8}) is valid. As a consequence, (\ref{3.7}) holds. Similar argument gives
 \[f(t,\eta, \eta') + \beta''(t) \leq a(\beta -\eta) - b(\beta' - \eta') \quad (\eta\in A_\Delta, t\in [0,1]). \]
These lead to that ($E_1$) is satisfied (with $\Delta=0$). By Theorem~\ref{thm 1}, problem (\ref{1.1}) has a solution $x(t)$ that satisfies $\al(t) \leq x(t)\leq \beta(t)$ and 
\[
\al'(t)-N(x(t)-\al(t)) \leq x'(t)\leq \beta'(t) +N(\beta(t)-x(t)),
\]
 where $N$ is a constant related to $\vep_1$ but independent of $\vep_2$. 

Step 2. Assume $\al(t) < \beta(t)$, $-\al''(t) \leq f(t,\al, \al')$, and $-\beta''(t) \geq f(t,\beta, \beta')$ for all $t\in [0,1]$. To each $\vep >0$ let
\[
f_\vep(t,x,y) = f(t,x,y) + \gamma_\vep(t,x),
\]
where 
\[
\gamma_\vep(t,x) = (1-\frac{x-\al(t)}{\beta(t)-\al(t)})\vep.
\]
 Clearly, $\gamma_\vep(t,\al(t))=\vep$, $\gamma_\vep(t,\beta(t))=-\vep$ and $|\gamma_\vep(t,x)|\leq\vep$ if $\al(t)\leq x\leq \beta(t)$. So we have $-\al''(t) < f_\vep(t,\al, \al')$, and $-\beta''(t) > f_\vep(t,\beta, \beta')$ for all $t\in [0,1]$. From the conclusion of step 1, for each $0<\vep <1$, there is $N>0$ independent of $\vep$ such that the problem
\begin{equation}\label{3.9}
 \begin{cases}
 -x'' = f_\vep(t,x,x') \qquad (t\in [0,1]) \\
 x(0)=x(1), \, x'(0) = x'(1).
 \end{cases}
 \end{equation}
 has a solution $x_\vep(t)$ satisfying $\al(t) \leq x_\vep(t)\leq \beta(t)$ and 
\[
\al'(t)-N(x_\vep(t)-\al(t)) \leq x_\vep'(t)\leq \beta'(t) +N(\beta(t)-x_\vep(t)).
\]
 Since $f$ is continuous, the nets $\{x_\vep(t)\}_{0<\vep<1}$, $\{x'_\vep(t)\}_{0<\vep<1}$
 and $\{x''_\vep(t)\}_{0<\vep<1}$ are uniformly bounded. From the Ascoli-Arzel\'a Theorem, there is a sequence $\vep_n \to 0+$ and a function $x(t)\in C^1[0,1]$ such that $\lim_{n\to\infty}x_{\vep_n}(t) = x(t)$
and $\lim_{n\to \infty}x'_{\vep_n}(t) = x'(t)$ uniformly on $[0,1]$. From the identity $-x''_\vep = f_\vep(t,x_\vep,x_\vep')$ we also have $\lim_{n\to \infty}x''_{\vep_n}(t)$ converges uniformly on  $[0,1]$. So $x''(t)$ exists and $-x'' = f_\vep(t,x,x')$ on $[0,1]$. Hence $x(t)$ is a solution of (\ref{1.1}) and  $\al(t) \leq x(t)\leq \beta(t)$ and 
\[
\al'(t)-N(x(t)-\al(t)) \leq x'(t)\leq \beta'(t) +N(\beta(t)-x(t)) \quad (t\in [0,1]).
\]

Step 3. We now prove Theorem~\ref{thm 2} for the general setting that $\al(t) \leq \beta(t)$, $-\al''(t) \leq f(t,\al, \al')$, and $-\beta''(t) \geq f(t,\beta, \beta')$ for all $t\in [0,1]$.

Let $F(t,x,y) = f(t,\bar x, y)$, where 
\[
\bar x = \max\{\al(t), \min\{x, \beta(t)\}\}.
\]
 For each $\vep > 0$ let $\al_\vep(t) = \al(t) - \vep$ and $\beta_\vep(t) = \beta(t)$. then $\al_\vep(t) < \beta_\vep(t)$ ($t\in [0,1]$). When $\vep$ is sufficiently small, ($E_2$) and ($E_3$) hold with $f$, $\al$ and $\beta$ replaced by $F$, $\al_\vep$ and $\beta_\vep$ respectively. From the conclusion obtained in step 2, the problem
\[
 \begin{cases}
 -x'' = F(t,x,x') \qquad (t\in [0,1]) \\
 x(0)=x(1), \, x'(0) = x'(1).
 \end{cases}
 \]
has a solution $x(t)$ satisfying 
\[
\al_\vep(t) \leq x(t)\leq \beta_\vep(t) = \beta(t).
\]
 We prove that $x(t) \geq \al(t)$ ($t\in [0,1]$). Then by the definition of $F$, $x(t)$ is indeed a solution of (\ref{1.1}) satisfying $\al(t) \leq x(t)\leq  \beta(t)$ ($t\in [0,1]$). This will complete the proof.

We may assume $\al(t)$ is not a solution of problem (\ref{1.1}). If $x(t)-\al(t) < 0$ for some $t\in [0,1]$ then $x(t) - \al(t)$ is not a constant. Otherwise $\al(t)$ would be a solution of (\ref{1.1}). Hence there is $t_1\in [0,1)$ such that $x(t_1)-\al(t_1) < 0$, $x'(t_1)=\al'(t_1)$, and in any right neighborhood of $t_1$ there exists $t$ such that $x(t)-\al(t) > x(t_1)-\al(t_1)$. This implies that in any right  neighborhood of $t_1$ there exists $t_2$ such that $x'(t_2)-\al'(t_2) >0$. On the other hand, since $x(t_1)<\al(t_1)$, there is a right  neighborhood of $t_1$ on which $x(t)<\al(t)$. So on this right neighborhood of $t_1$
\[
-x''(t) = f(t,\al(t),x'),
\]
that is, on this neighborhood, $x'(t)$ is a solution of the initial value problem
\begin{equation}\label{3.10}
-y'(t) = f(t,\al(t),y), \quad y(t_1) = \al'(t_1). 
\end{equation}
By ($E_2$), $f(t,\al(t),y)$ is locally Lipschitzian with respect to $y$ near the curve $(t,\al(t),\al'(t))$. Hence the solution of (\ref{3.10}) is locally unique. However, $\al(t)$ is a lower solution of (\ref{1.1}). It satisfies 
\[
-(\al'(t))' \leq f(t,\al(t),\al').
\]
By a well-known differential inequality (see, for instant, \cite[Theorem~III.4.1]{18}), we must have $\al'(t) \geq x'(t)$ in an entire right neighborhood of $t_1$. This contradicts the property of those points $t_2$. Therefore $x(t) \geq \al(t)$ ($t\in [0,1]$) as long as $\al(t)$ is not a solution of problem (\ref{1.1}). But if  $\al(t)$ is a solution of problem (\ref{1.1}), then the result is trivial. The proof is complete.
\qed

\end{document}